\documentclass{amsart}
\usepackage{amsmath}
\usepackage{amssymb}
\usepackage{amsfonts}

\newtheorem{theorem}{Theorem}
\theoremstyle{plain}

\newtheorem{corollary}{Corollary}

\newtheorem{definition}{Definition}
\newtheorem{example}{Example}

\newtheorem{proposition}{Proposition}

\numberwithin{equation}{section}
\input{tcilatex}

\begin{document}
\title[On Quasi $n$-absorbing Elements of Multiplicative Lattices]{On Quasi $%
n$-absorbing Elements of Multiplicative Lattices}
\author{ Ece Yetkin Celikel}
\address{Department of Mathematics, Gaziantep University, Gaziantep, Turkey}
\email{yetkin@gantep.edu.tr}
\date{March 2, 2016}
\subjclass[2000]{13A15}
\keywords{2-absorbing elements, n-absorbing elements, multiplicative
lattices, quasi $n$-absorbing elements}

\begin{abstract}
In this study, we introduce the concept of quasi $n$-absorbing elements of
multiplicative lattices. A proper element $q$ is said to be a quasi $n$%
-absorbing element of $L$ if whenever $a^{n}b\leq q$ implies that either $%
a^{n}\leq q$ or $a^{n-1}b\leq q$. We investigate properties of this new type
of elements and obtain some relations among prime, 2-absorbing, $n$%
-absorbing elements in multiplicative lattices.
\end{abstract}

\maketitle

\section{Introduction}

In this paper we define and study quasi $n$-absorbing elements in
multiplicative lattices. A multiplicative lattice is a complete lattice $L$
with the least element $0$ and compact greatest element $1,$ on which there
is defined a commutative, associative, completely join distributive product
for which $1$ is a multiplicative identity. Notice that $L(R)$ the set of
all ideals of a commutative ring $R$ is a special example for multiplicative
lattices which is principally generated, compactly generated and modular.
However, there are several examples of non-modular multiplicative lattices
(see \cite{fd}) Weakly prime ideals \cite{ds} were generalized to
multiplicative lattices by introducing weakly prime elements in \cite{fct}.
While 2-absorbing, weakly 2-absorbing ideals in commutative rings, and $n$%
-absorbing ideals were introduced in \cite{b}, \cite{bd}, and \cite{Dand},
2-absorbing and weakly 2-absorbing elements in multiplicative lattices were
studied in \cite{JTY}.

We begin by recalling some background material which will be needed. An
element $a$ of $L$ is said to be compact if whenever $a\leq
\dbigvee\limits_{\alpha \in I}a_{\alpha }$ implies $a\leq
\dbigvee\limits_{\alpha \in I_{0}}a_{\alpha }$ for some finite subset $I_{0}$
of $I$. By a $C$-lattice we mean a (not necessarily modular) multiplicative
lattice which is generated under joins by a multiplicatively closed subset $%
C $ of compact elements of $L$. We note that in a $C$-lattice, a finite
product of compact elements is again compact. Throughout this paper $L$ and $%
L_{\ast }$ denotes a multiplicative lattice and the set of compact elements
of the lattice $L$, respectively. An element $a$ of $L$ is said to be proper
if $a<1.$ A proper element $p$ of $L$ is said to be prime (resp. weakly
prime) if $ab\leq p$ (resp. $0\neq ab\leq p)$ implies either $a\leq p$ or $%
b\leq p.$ If $0$ is prime, then $L$ is said to be a domain. A proper element 
$m$ of $L$ is said to be maximal if $m<x\leq 1$ implies $x=1.$ The jacobson
radical of a lattice $L$ is defined as $J(L)=\wedge \{m$ $:$ $m$ is a
maximal element of $L\}.$ $L$ is said to be quasi-local if it contains a
unique maximal element. If $L=\left\{ 0,1\right\} $, then $L$ is called a
field. For $a\in L,$ we define radical of $a$ as $\sqrt{a}=\wedge \{p\in L:p$
is prime and $a\leq p\}.$ Note that in a $C$-lattice $L$, $\sqrt{a}=\wedge
\{p\in L:p$ is prime and $a\leq p\}=\vee \{x\in L_{\ast }\mid x^{n}\leq a$
for some $n\in Z^{+}\}$ by (Theorem 3.6 of \cite{nk}). The nilpotent
elements of $L$ is the set of$\ Nil(L)=\sqrt{0}$. $C$-lattices can be
localized. For any prime element $p$ of $L$, $L_{p}$ denotes the
localization at $F$ $=$ $\{x\in C\mid x\nleq p\}$. For details on $C$%
-lattices and their localization theory, the reader is referred to \cite{jw}
and \cite{js}. An element $e\in L$ is said to be principal \cite{dil}, if it
satisfies the dual identities (i) $a\wedge be=((a:e)\wedge b)e$ and (ii) $%
(ae\vee b):e=(b:e)\vee a$. Elements satisfying the identity (i) are called
meet principal and elements satisfying the identity (ii) are called join
principal. Note that by \cite[Lemma 3.3 and Lemma 3.4]{dil}, a finite
product of meet (join) principal elements of $L$ is again meet (join)
principal. If every element of $L$ is principal, then $L$ is called a
principal element lattice which is studied in \cite{dd2}.

Recall from \cite{JTY} that a proper element $q$ of $L$ is called a $2$%
-absorbing (resp. weakly $2$-absorbing) element of $L$ if whenever $a,b,c\in
L$ with $abc\leq q$ (resp. $0\neq abc\leq q),$ then either $ab\leq q$ or $%
ac\leq q$ or $bc\leq q$. We say that $\ (a,b,c)$ is a triple zero element of 
$q$ if $abc=0$, $ab\not\leq q$, $ac\not\leq q$ and $bc\not\leq q$. Observe
that if $q$ is a\ weakly 2-absorbing element which is not a 2-absorbing,
then there exist a triple zero of $q.$ As a generalization of 2-absorbing
elements, a proper element $q$ of $L$ is called a $n$-absorbing (resp.
weakly $n$-absorbing) element of $L$ if whenever $a_{1}a_{2}...a_{n+1}\leq q$
$($ resp. $0\neq a_{1}a_{2}...a_{n+1}\leq q)$ for some $a_{1}a_{2}...a_{n+1}%
\in L_{\ast }$ implies that $a_{1}a_{2}...a_{k-1}a_{k+1}...a_{n+1}\leq q$
for some $k=1,...,n+1.$

\section{Quasi $n$-absorbing and weakly quasi $n$-absorbing elements}

Let $L$ be a multiplicative lattice and $n$ be a positive integer. In this
section, we introduce quasi $n$-absorbing and weakly quasi $n$-absorbing
elements of $L$ and investigate their basic properties.

\begin{definition}
\label{d1}Let $q$ be a proper element of $L.$
\end{definition}

\begin{enumerate}
\item $q$ is said to be a quasi $n$-absorbing element of $L$ if whenever $%
a^{n}b\leq q$ for some $a,b\in L_{\ast }$ implies that either $a^{n}\leq q$
or $a^{n-1}b\leq q.$

\item $q$ is said to be a weakly quasi $n$-absorbing element of $L$ if
whenever $0\neq a^{n}b\leq q$ for some $a,$ $b\in L_{\ast }$ implies that
either $a^{n}\leq q$ or $a^{n-1}b\leq q$.
\end{enumerate}

We can obtain some relations among the concepts of prime, 2-absorbing, $n$%
-absorbing, quasi $n$-absorbing elements and weakly quasi $n$-absorbing
elements by the following:

\begin{theorem}
\label{relations}Let $q$ be a proper element of $L$ and $n\geq 1.$ Then the
following statements hold:
\end{theorem}

\begin{enumerate}
\item $q$ is a prime element of $L$ if and only if $\ q$ is a quasi $1$%
-absorbing element of $L.$

\item $q$ is a weakly prime element of $L$ if and only if $\ q$ is a weakly
quasi $1$-absorbing element of $L.$

\item If $q$ is $n$-absorbing, then $q$ is a quasi $n$-absorbing element of $%
L$.

\item If $q$ is quasi $n$-absorbing, then $q$ is a weakly quasi $n$%
-absorbing element of $L$.

\item If $q$ is a quasi $n$-absorbing element of $L$, then $q$ is a quasi $m$%
-absorbing element of $L$ for all $m\geq n.$

\item If $q$ is weakly quasi $n$-absorbing, then $q$ is a weakly quasi $m$%
-absorbing element of $L$ for all $m\geq n.$
\end{enumerate}

\begin{proof}
(1) and (2) are clear from definition \ref{d1}. Since (4) and (6) can be
shown very similar to (3) and (5), respectively, it is sufficient to prove
(3) and (5) only.

(3) Suppose that $q$ is a $n$-absorbing element of $L$ and assume that $%
a^{n}b\leq q$ for some $a,b\in L_{\ast }.$ Then $\underset{n\text{ times}}{%
\underbrace{a...a}b}$ $\leq q$ implies that either $a^{n}\leq q$ or $%
a^{n-1}b\leq q$, we are done.

(5) Suppose that $q$ is a quasi $n$-absorbing element of $L,$ and let $%
a,b\in L_{\ast }$ with $a^{m}b\leq q$ for some $m\geq n$. Hence $%
a^{n}(a^{m-n}b)\leq q$. Since $q$ is quasi $n$-absorbing, we have either $%
a^{n}\leq q$ or $a^{n-1}(a^{m-n}b)\leq q$. It means either $a^{m}\leq q$ or $%
a^{m-1}b\leq q$. This shows that $q$ is a quasi $m$-absorbing element of $L.$
\end{proof}

We have the following corollary as a direct result of Theorem \ref{relations}%
:

\begin{corollary}
\label{c1}Let $q$ be a proper element of $L.$
\end{corollary}

\begin{enumerate}
\item If $q$ is a prime element of $L$, then $q$ is a quasi $n$-absorbing
element of $L$ for all $n\geq 1.$

\item If $q$ is weakly prime, then $q$ is a weakly quasi $n$-absorbing
element of $L$ for all $n\geq 1.$

\item If $q$ is a 2-absorbing element of $L,$ then $q$ is a quasi $n$%
-absorbing element of $L$ for all $n\geq 2.$

\item If $q$ is weakly 2-absorbing$,$ then $q$ is a weakly quasi $n$%
-absorbing element of $L$ for all $n\geq 2.$
\end{enumerate}

\begin{proof}
(1), (2) Since a (weakly) prime element of $L$ is exactly a (weakly) quasi
1-absorbing element of $L$, this result is clear by \ref{relations} (5), (6).

(3), (4) Since a (weakly) 2-absorbing element of $L$ is a (weakly) quasi
2-absorbing element of $L$, so we are done by again (5), (6).
\end{proof}

However the converses of these relations above are not true in general.

\begin{example}
Consider the lattice of ideals of the ring of integers $L=L(%
%TCIMACRO{\U{2124} }%
%BeginExpansion
\mathbb{Z}
%EndExpansion
)$. Note that the element $30%
%TCIMACRO{\U{2124} }%
%BeginExpansion
\mathbb{Z}
%EndExpansion
$ of $L$ is a quasi $2$-absorbing element, and so quasi $n$-absorbing
element for all $n\geq 2$ by Corollary \ref{c1}, but it is not a 2-absorbing
element of $L$ by Theorem 2.6 in \cite{fct}.
\end{example}

\begin{proposition}
Let $q$ be a proper element of $L$. Then the following statements are
equivalent:
\end{proposition}

\begin{enumerate}
\item $q$ is a quasi $n$-absorbing element of $L$.

\item $(q:a^{n})=(q:a^{n-1})$ where $a\in L_{\ast },$ $a^{n}\nleq q$.
\end{enumerate}

In paticular $0$ is a quasi $n$-absorbing element of $L$ if and only if$\ $%
for each $a\in L_{\ast }$, $a^{n}=0$ or $ann(a^{n})=ann(a^{n-1}).$

\begin{proof}
It is clear from Definition \ref{d1}.
\end{proof}

Notice that if $q$ is a weakly quasi $n$-absorbing element which is not
quasi $n$-absorbing, then there are some elements $a,$ $b\in L_{\ast }$ such
that $a^{n}b=0$, $a^{n}\nleq q$ and $a^{n-1}b\nleq q$. We call the elements $%
(a,b)$ with this property as a quasi $n$-zero element of $q$. Notice that a
zero divisor element of $L$ is a quasi $1$-zero element of $0_{L},$ and $%
(a,a,b)$ is a triple zero element of $q$ if and only if $(a,b)$ is a quasi
2-zero element of $q.$

\begin{theorem}
Let $q$ be a weakly quasi $n$-absorbing element of $L$. If $(a,b)$ is a
quasi $n$-zero element of $q$ for some $a,b\in L_{\ast },$ then $a^{n}\in
ann(q)$ and $b^{n}\in ann(q).$
\end{theorem}

\begin{proof}
Suppose that $a^{n}\notin ann(q).$ Hence $a^{n}q_{1}\neq 0$ for some$\
q_{1}\in L_{\ast }$ where $q_{1}\leq q.$ It follows $0\neq a^{n}(b\vee
q_{1})\leq q$. Since $a^{n}\nleq q$, and $q$ is weakly quasi $n$-absorbing,
we conclude that $a^{n-1}(b\vee q_{1})\leq q.$ So $a^{n-1}b\leq q,$ a
contradiction. Thus $a^{n}q=0$, and so $a^{n}\in ann(q)$. Similarly we
conclude that $b^{n}\in ann(q)$.
\end{proof}

\begin{theorem}
\label{pintersect}
\end{theorem}

\begin{enumerate}
\item Let $\{p_{\lambda }\}_{\lambda \in \Lambda }$ be a family of prime
elements of $L,$ then $\underset{\lambda \in \Lambda }{\overset{}{\tbigwedge 
}}p_{\lambda }$ is a quasi $m$-absorbing element for all $m\geq 2.$

\item Let $\{p_{\lambda }\}_{\lambda \in \Lambda }$ be a family of weakly
prime elements of $L,$ then $\underset{\lambda \in \Lambda }{\overset{}{%
\tbigwedge }}p_{\lambda }$ is a weakly quasi $m$-absorbing element for all $%
m\geq 2.$
\end{enumerate}

\begin{proof}
(1) To prove this argument above, we need to verify that $\underset{\lambda
\in \Lambda }{\overset{}{\tbigwedge }}p_{\lambda }$ is a quasi 2-absorbing
element of $L$ by Corollary \ref{c1} (3). Let $a,b\in L_{\ast }$ with $%
a^{2}b\leq \underset{\lambda \in \Lambda }{\overset{}{\tbigwedge }}%
p_{\lambda }.$ Since $a^{2}b\leq p_{i}$ for all $p_{i}$ prime elements, we
have $a\leq p_{i}$ or $b\leq p_{i}.$ Thus $ab\leq p_{i}$ for all $i=1,...,n$
and so $ab\leq $ $\underset{\lambda \in \Lambda }{\overset{}{\tbigwedge }}%
p_{\lambda }$, which completes the proof.

(2) It can be easily obtained that $\underset{\lambda \in \Lambda }{\overset{%
}{\tbigwedge }}p_{\lambda }$ is a weakly quasi 2-absorbing element by the
similar argument in (1). Consequently it is a weakly quasi $m$-absorbing
element$\ $or all $m\geq 2$ by Corollary \ref{c1} (4).
\end{proof}

\begin{corollary}
Let $q$ be a proper element of $L.$ Then $\sqrt{q},$ $Nil(L)$ and $Jac(L)$
are quasi $n$-absorbing elements of $L$ for all $n\geq 2.$
\end{corollary}

\begin{proof}
It is clear from Theorem \ref{pintersect}.
\end{proof}

\begin{theorem}
Let $L$ be a totally ordered lattice and $m$ a positive integer.
\end{theorem}

\begin{enumerate}
\item If $\{q_{\lambda }\}_{\lambda \in \Lambda }$ is a family of quasi $m$%
-absorbing elements of $L,$ then $\underset{\lambda \in \Lambda }{\overset{}{%
\tbigwedge }}q_{\lambda }$ is a quasi $m$-absorbing element of $L$.

\item If $\{q_{\lambda }\}_{\lambda \in \Lambda }$ is a family of weakly
quasi $m$-absorbing elements of $L,$ then $\underset{\lambda \in \Lambda }{%
\overset{}{\tbigwedge }}q_{\lambda }$ is a weakly quasi $m$-absorbing
element of $L.$
\end{enumerate}

\begin{proof}
(1) Assume that $\{q_{\lambda }\}_{\lambda \in \Lambda }$ is an ascending
chain and $a^{m}\nleq \underset{\lambda \in \Lambda }{\overset{}{\tbigwedge }%
}q_{\lambda }$and $a^{m-1}b\nleq \underset{\lambda \in \Lambda }{\overset{}{%
\tbigwedge }}q_{\lambda }$. We show that $a^{m}b\nleq \underset{\lambda \in
\Lambda }{\overset{}{\tbigwedge }}q_{\lambda }$. Hence $a^{m}\nleq q_{j}$
and $a^{m-1}b\nleq q_{k}$ for some $j,k=1,...,n.$ Put $t=\min \{j,k\}$. Then 
$a^{m}\nleq q_{t}$ and $a^{m-1}b\nleq q_{t}$. Since $q_{t}$ is a quasi $m$%
-absorbing element, it follows $a^{m}b\nleq q_{t}$. Thus $a^{m}b\nleq 
\underset{\lambda \in \Lambda }{\overset{}{\tbigwedge }}q_{\lambda }$, we
are done.

(2) It can be easily shown similar to (1).
\end{proof}

\begin{enumerate}
\item If $q_{1},$...$,q_{n}$ are quasi $m_{i}$-absorbing elements of $L$ for
all $i=1,2,...,n$, then $\underset{i=1}{\overset{n}{\tbigwedge }}q_{i}$ is a
quasi $m$-absorbing element of $L$ where $m=\max \{m_{1},...,m_{n}\}+1.$

\item If $q_{1},$...$,q_{n}$ are weakly quasi $m_{i}$-absorbing elements of $%
L$ for all $i=1,2,...,n,$ then $\underset{i=1}{\overset{n}{\tbigwedge }}%
q_{i} $ is a weakly quasi $m$-absorbing element of $L$ where $m=\max
\{m_{1},...,m_{n}\}+1.$
\end{enumerate}

\begin{proof}
(1) Let $a,b\in L_{\ast }$ such that $a^{m}b\leq $ $\underset{i=1}{\overset{n%
}{\tbigwedge }}q_{i}.$ Hence $a^{m_{i}}\leq q_{i}$ or $a^{m_{i}-1}b\leq
q_{i} $ for all $i=1,..,n.$ Now assume that $a^{m}\not\leq \underset{i=1}{%
\overset{n}{\tbigwedge }}q_{i}.$ Without loss generality we can suppose that 
$a^{m_{i}}\leq q_{i}$ for all $1\leq i\leq j,$ and $a^{m_{i}}\not\leq q_{i}$
for all $j+1\leq i\leq n.$ Hence we have $a^{m_{i}-1}b\leq q_{i}$ for all $%
j+1\leq i\leq n.$ Then we get clearly $a^{m-1}b\leq q_{i}$ for $m=\max
\{m_{1},...,m_{n}\}+1$ and for all $1\leq i\leq n.$ Thus $a^{m-1}b\leq 
\underset{i=1}{\overset{n}{\tbigwedge }}q_{i}$, so we are done.

(2) The proof is obtained similar to (1).
\end{proof}

If $x\in L,$ the interval $[x,1]$ is denoted $L/x.$ The elemets of $%
\overline{a}$ and $L/x$ is again a multiplicative lattice with $\overline{a}%
\circ \overline{b}=ab\vee x$ for all $\overline{a},\overline{b}\in L/x.$

\begin{theorem}
Let $x$ and $q$ be proper elements of $L$ with $x\leq q.$
\end{theorem}

\begin{enumerate}
\item If $q$ is a quasi $n$-absorbing element of $L,$ then $\overline{q}$ is
a quasi $n$-absorbing element of $L/x.$

\item If $q$ is a weakly quasi $n$-absorbing element of $L,$ then $\overline{%
q}$ is a weakly quasi $n$-absorbing element of $L/x.$
\end{enumerate}

\begin{proof}
(1) Suppose that $\overline{a}=a\vee x,$ $\overline{b}=b\vee x\in L$ with $%
\overline{a}^{n}\overline{b}\leq \overline{q}.$ Then $a^{n}b\vee x\leq q$,
and so $a^{n}b\leq q.$ Since $q$ is quasi 2-absorbing, we get either $%
a^{n}\leq q$ or $a^{n-1}b\leq q$. Thus $\overline{a}^{n}=(a\vee x)^{n}\leq 
\overline{q}$ or $\overline{a}^{n-1}\overline{b}=\left( a\vee x\right)
^{n-1}(b\vee x)\leq \overline{q}$, as needed.

(2) Let $\overline{a}=a\vee x,$ $\overline{b}=b\vee x\in L$ with $x\neq 
\overline{a}^{n}\overline{b}\leq \overline{q}.$ Then $x\neq a^{n}b\vee x\leq
q$, and so $0\neq a^{n}b\leq q.$ Since $q$ is weakly quasi 2-absorbing, we
conclude that either $a^{n}\leq q$ or $a^{n-1}b\leq q$. Consequently, $%
\overline{a}^{n}=(a\vee x)^{n}\leq \overline{q}$ or $\overline{a}^{n-1}%
\overline{b}=\left( a\vee x\right) ^{n-1}(b\vee x)\leq \overline{q}$, we are
done.
\end{proof}

Recall that any $C$-lattice can be localized at a multiplicatively closed
set. Let $L$ be a $C$-lattice and $S$ a multiplicatively closed subset of $%
L_{\ast }$.Then for $a\in L,$ $a_{S}=\vee \{x\in L_{\ast }:$ $xs\leq a$ for
some $s\in S\}$ and $L_{S}=\{a_{S}:a\in L\}.$ $L_{S}$ is again a
multiplicative lattice under the same order as $L$ with the product $%
a_{S}\circ b_{s}=(a_{S}b_{S})_{S}$ where the right hand side is evaluated in 
$L$. If $p\in L$ is prime and $S=\{x\in L_{\ast }:x\nleq p\}$, then $L_{S}$
is denoted by $L_{p}.$ \cite{jw}

\begin{theorem}
Let $m$ be a maximal element of $L$ and $q$ be a proper element of $L.$
\end{theorem}

\begin{enumerate}
\item If $q$ is a quasi $n$-absorbing element of $L,$ then $q_{m}$ is a
quasi $n$-absorbing element of $L_{m}.$

\item If $q$ is a weakly quasi $n$-absorbing element of $L,$ then $q_{m}$ is
a weakly quasi $n$-absorbing element of $L_{m}.$
\end{enumerate}

\begin{proof}
(1) Let $a,b\in L_{\ast }$ such that $a_{m}^{n}b_{m}\leq q_{m}$. Hence $%
ua^{n}b\leq q$ for some $u\nleq m.$ It implies that $a^{n}\leq q$ or $%
a^{n-1}(ub)\leq q.$ Since $u_{m}=1_{m},$ we get $a_{m}^{n}\leq q_{m}$ or $%
a_{m}^{n-1}b_{m}\leq q_{m}$, we are done.

(2) It can be verified by the similar argument in (1).
\end{proof}

\begin{theorem}
\label{t1}Let$\ L$ be a principal element lattice. Then the following
statements are equivalent:
\end{theorem}

\begin{enumerate}
\item Every proper element of $L$ is a quasi $n$-absorbing element of $L.$

\item For every $a,$ $b\in L_{\ast }$, $a^{n}=ca^{n}b$ or $a^{n-1}b=da^{n}b$
for some $c,$ $d\in L$.

\item For all $a_{1},a_{2},...,a_{n+1}\in L_{\ast },$ $(a_{1}\wedge
a_{2}\wedge ...\wedge a_{n})^{n}\leq ca_{1}a_{2}...a_{n+1}$ or $(a_{1}\wedge
a_{2}\wedge ...\wedge a_{n})^{n-1}a_{n+1}\leq da_{1}a_{2}...a_{n+1}$ for
some $c,$ $d\in L$.
\end{enumerate}

\begin{proof}
(1)$\Leftrightarrow $(2) Suppose that every proper element of $L$ is a quasi 
$n$-absorbing element of $L.$ Hence $a^{n}b\leq (a^{n}b)$ implies that $%
a^{n}\leq (a^{n}b)$ or $a^{n-1}b\leq (a^{n}b)$. Since $L$ is a principal
element lattice, there is some element $c\in L$ with $a^{n}=ca^{n}b$ or
there is some element $d\in L$ with $a^{n-1}b=da^{n}b$. The converse is
clear.

(2)$\Rightarrow $(3) Put $a=a_{1}\wedge a_{2}\wedge ...\wedge a_{n}$ and $%
b=a_{n+1}.$ Hence the result follows from (2).

(3)$\Rightarrow $(2) For all $a,b\in L_{\ast },$ we can write $a^{n}=%
\underset{n\text{ times}}{(\underbrace{a\wedge a\wedge ...\wedge a})}$ $\leq
ca^{n}b$ or $a^{n-1}b=\underset{n-1\text{ times}}{(\underbrace{a\wedge
a\wedge ...\wedge a})b}$ $\leq da^{n}b.$
\end{proof}

\begin{theorem}
Let $L=L_{1}\times L_{2}$ where $L_{1}$ and $L_{2}$ are $C$-lattices. Then
the following statements hold:
\end{theorem}

\begin{enumerate}
\item $q_{1}$ is a quasi $n$-absorbing element of $L_{1}$ if and only if $%
(q_{1},1_{L_{2}})$ is a quasi $n$-absorbing element of $L.$

\item $q_{2}$ is a quasi $n$-absorbing element of $L_{2}$ if and only if $%
(1_{L_{1}},q_{2})$ is a quasi $n$-absorbing element of $L.$
\end{enumerate}

\begin{proof}
(1) Suppose that $q_{1}$ is a quasi $n$-absorbing element of $L_{1}.$ Let $%
(a_{1},a_{2})^{n}(b_{1},b_{2})\leq (q_{1},1_{L_{2}})$ for some $a_{1},$ $%
b_{1}\in L_{1_{\ast }}$ and $a_{2},$ $b_{2}\in L_{2_{\ast }}$.$\ $Then $%
a_{1}^{n}b_{1}\leq q_{1}$ implies that either $a_{1}^{n}\leq q_{1}$ or $%
a_{1}^{n-1}b_{1}\leq q_{1}.$ It follows either $(a_{1},a_{2})^{n}\leq
(q_{1},1_{L_{2}})$ or $(a_{1},a_{2})^{n-1}(b_{1},b_{2})\leq
(q_{1},1_{L_{2}}).$ Thus $(q_{1},1_{L_{2}})$ is a quasi $n$-absorbing
element of $L.$ Conversely suppose that $(q_{1},1_{L_{2}})$ is a quasi $n$%
-absorbing element of $L$ and $a^{n}b\leq q_{1}$ for some $a,b\in L_{1_{\ast
}}$. Hence $(a,1_{L_{2}})^{n}(b,1_{L_{2}})\leq (q_{1},1_{L_{2}})$ which
implies that either $(a,1_{L_{2}})^{n}\leq (q_{1},1_{L_{2}})$ or $%
(a,1_{L_{2}})^{n-1}(b,1_{L_{2}})\leq (q_{1},1_{L_{2}}).$ So $a_{1}^{n}\leq
q_{1}$ or $a_{1}^{n-1}b_{1}\leq q_{1}$, as needed.

(2) It can be verified easily similar to (1).
\end{proof}

\begin{theorem}
\label{tcartes}Let $L=L_{1}\times ...\times L_{k}$ where $L_{i}^{{}}$ s are $%
C$-lattices for all $i=1,...,k$. If $q_{i}$ is a quasi $n_{i}$-absorbing
element of $L_{i}$ for all $i=1,...,k,$ then $(q_{1},...,q_{k})$ is a quasi $%
m$-absorbing element of $L$ where $m=\max \{n_{1},$...,$n_{k}\}+1.$
\end{theorem}

\begin{proof}
Suppose that $(a_{1},...,a_{k})^{m}(b_{1},...,b_{k})\leq (q_{1},...,q_{k})$
for $(a_{1},...,a_{k}),(b_{1},...,b_{k})\in L_{\ast }$ and $m=\max \{n_{1},$%
...,$n_{k}\}+1$. Hence $a_{i}^{m}b_{i}=a_{i}^{n_{i}}(a_{i}^{m-n_{i}}b_{i})%
\leq q_{i}$ for all $i=1,..,k$. Since each $q_{i}$ is a quasi $n_{i}$%
-absorbing element, we have either $a_{i}^{n_{i}}\leq q_{i}$ or $%
a_{i}^{m-1}b_{i}=a_{i}^{n_{i}-1}(a_{i}^{m-n_{i}}b_{i})\leq q_{i}$ for all $%
i=1,..,k$. If $a_{i}^{n_{i}}\leq q_{i}$ for all $i=1,...,k,$ then $%
(a_{1},...,a_{k})^{m}\leq (q_{1},...,q_{k}).$ Without loss generality,
suppose that $a_{i}^{n_{i}}\leq q_{i}$ for all $1\leq i\leq j$ and $%
a_{i}^{m-1}b_{i}\leq q_{i}$ for all $j+1\leq i\leq k,$ for some $j=1,..k.$
Thus $(a_{1},...,a_{k})^{m-1}(b_{1},...,b_{k})\leq (q_{1},...,q_{k}),$ so we
are done.
\end{proof}

\begin{definition}
\label{d2}A proper element $q$ of $L$ is said to be a strongly quasi $n$%
-absorbing element of $L$ if whenever $a,b\in L$ (not necessarily compact)
with $a^{n}b\leq q$ implies that either $a^{n}\leq q$ or $a^{n-1}b\leq q$.
\end{definition}

It is clearly seen that every strongly quasi $n$-absorbing element of $L$ is
quasi $n$-absorbing.

\begin{theorem}
Let$\ L$ be a principal element lattice. The following statements are
equivalent:
\end{theorem}

\begin{enumerate}
\item Every proper element of $L$ is a strongly quasi $n$-absorbing element
of $L.$

\item For all $a,b\in L$, $a^{n}=a^{n}b$ or $a^{n-1}b=a^{n}b$,

\item $(a_{1}\wedge a_{2}\wedge ...\wedge a_{n})^{n}\leq
a_{1}a_{2}...a_{n+1} $ or $(a_{1}\wedge a_{2}\wedge ...\wedge
a_{n})^{n-1}a_{n+1}\leq a_{1}a_{2}...a_{n+1}$ for all $a_{1},$ $a_{2},$ $%
...,a_{n+1}\in L.$
\end{enumerate}

\begin{proof}
This can be easily shown using the similar argument in Theorem \ref{t1}.
\end{proof}

\begin{theorem}
Let $q$ be a proper element of $L.$ Then the following statements hold:
\end{theorem}

\begin{enumerate}
\item If $a^{n}b\leq q\leq a\wedge b$ where $a,b\in L$ implies that $%
a^{n}\leq q$ or $a^{n-1}b\leq q,$ then $q$ is a strongly quasi $n$-absorbing
element of $L.$

\item If $a_{1}a_{2}...a_{n+1}\leq q\leq a_{1}\wedge a_{2}\wedge ...\wedge
a_{n+1}$ where $a_{1},a_{2},...,a_{n+1}\in L$ implies that $%
a_{1}...a_{i-1}a_{i+1}...a_{n+1}\leq q,$ for some $1\leq i\leq n+1$ then $q$
is a strongly quasi $n$-absorbing element of $L.$
\end{enumerate}

\begin{proof}
(1) Let $x,y\in L$ with $x^{n}y\leq q$. We show that $x^{n}\leq q$ or $%
x^{n-1}y\leq q.$ Now put $a=x\vee q$ and $b=y\vee q.$ Hence we conclude $%
a^{n}b\leq q\leq a\wedge b,$ and so $a^{n}\leq q$ or $a^{n-1}b\leq q$ by
(1). It follows $x^{n}\leq q$ or $x^{n-1}y\leq q.$

(2) It can be easily verified similar to (1).
\end{proof}


\begin{thebibliography}{99}
\bibitem{fd} F. Alarcon and D. D. Anderson, Commutative semirings and their
lattices of ideals, Houston Journal of Mathematics, \textbf{20} (1994),
571-590.

\bibitem{dd2} D. D. Anderson and C. Jayaram, Principal element lattices,
Czechoslovak Mathematical Journal, \textbf{46} (1996), 99-109 .

\bibitem{ds} D. D. Anderson and E. Smith, Weakly prime ideals, Houston
Journal of Mathematics, \textbf{29} (2003), 831-840.

\bibitem{Dand} D. F. Anderson, A. Badawi, On $n$-absorbing ideals of
commutative rings, Communications in Algebra, \textbf{39 }(2011), 1646-1672.

\bibitem{b} A. Badawi, On 2-absorbing ideals of commutative rings, Bull.
Austral. Math. Soc., \textbf{75} (2007), 417-429.

\bibitem{bd} A. Badawi and A.Y. Darani, On weakly 2-absorbing ideals of
commutative rings, Houston Journal of Mathematics, \textbf{39} (2013),
441-452.

\bibitem{fct} F. Callialp, C. Jayaram, U. Tekir, Weakly prime elements in
multiplicative lattices, Communications in Algebra, \textbf{40} (2012),
2825-2840.

\bibitem{dil} R. P. Dilworth, Abstract commutative ideal theory, Pacific
Journal of Mathematics, \textbf{12 }(1962), 481-498.

\bibitem{jw} C. Jayaram, E.W Johnson, s-prime elements in multiplicative
lattices. Periodica Mathematica Hungarica. \textbf{31 }(1995), 201-208 .

\bibitem{JTY} C. Jayaram, U. Tekir, E. Yetkin, 2-absorbing and weakly
2-absorbing elements in multiplicative lattices, Communications in Algebra, 
\textbf{42 }(2014), 2338-2353.

\bibitem{js} J.A. Johnson and G. R. Sherette, Structural properties of a new
class of CM-lattices. Canadian Journal of Mathematics. \textbf{38} (1986)
552-562.

\bibitem{nk} N. K. Thakare, C. S. Manjarekar, S. Maeda, Abstract spectral
theory II, Minimal characters and minimal spectrums of multiplicative
lattices. Acta. Sci. Math. (Szeged). \textbf{52 }(1988), 53-67
\end{thebibliography}
\end{document}